\documentclass[12pt,amstex]{amsart}
\usepackage{amssymb,amsthm,amsmath,ulem, fullpage}

\DeclareFontFamily{OMS}{rsfs}{\skewchar\font'60}
\DeclareFontShape{OMS}{rsfs}{m}{n}{<-5>rsfs5 <5-7>rsfs7 <7->rsfs10 }{}
\DeclareSymbolFont{rsfs}{OMS}{rsfs}{m}{n}
\DeclareSymbolFontAlphabet{\scr}{rsfs}

\newtheorem{theorem}{Theorem}[section]
\newtheorem{lemma}[theorem]{Lemma}

\theoremstyle{definition}
\newtheorem{definition}[theorem]{Definition}

\theoremstyle{remark}
\newtheorem{remark}[theorem]{Remark}

\newtheorem{question}[theorem]{Question}

\theoremstyle{theorem}
\newtheorem{proposition}[theorem]{Proposition}
\theoremstyle{theorem}
\newtheorem{corollary}[theorem]{Corollary}

\theoremstyle{definition}

\input xy
\xyoption{all}

\newcommand{\blank}{\underline{\hskip 10pt}}
\renewcommand{\O}{\mbox{$\mathcal{O}$}}
\newcommand{\Cech}{{$\check{\text{C}}$ech} }
\newcommand{\nsubset}{\not\subset}
\newcommand{\sI}{\scr{I}}
\newcommand{\sH}{\scr{H}}
\newcommand{\sG}{\scr{G}}
\newcommand{\sF}{\scr{F}}
\newcommand{\mydot}{{{\,\begin{picture}(1,1)(-1,-2)\circle*{2}\end{picture}\ }}}
\newcommand{\qis}{\simeq_{\text{qis}}}
\newcommand{\myR}{{\bf R}}
\newcommand{\DuBois}[1]{{\underline \Omega {}^0_{#1}}}
\newcommand{\FullDuBois}[1]{{\underline \Omega {}^{\mydot}_{#1}}}
\newcommand{\bA}{\mathbb{A}}
\newcommand{\bC}{\mathbb{C}}
\newcommand{\bF}{\mathbb{F}}
\newcommand{\bH}{\mathbb{H}}
\newcommand{\bQ}{\mathbb{Q}}
\newcommand{\bZ}{\mathbb{Z}}
\newcommand{\SIClo}[2]{\hskip 3pt {\vphantom{#1}}^{+}_{\hskip -3pt #2}{\hskip -3pt #1}} 
\newcommand{\WSIClo}[2]{\hskip 4pt {\vphantom{#1}}^{*}_{\hskip -4pt #2}{\hskip -2pt #1}} 
\newcommand{\SN}[1]{\hskip -2pt {\vphantom{#1}}^{+}{\hskip -3pt #1}} 
\newcommand{\WN}[1]{\hskip -1pt {\vphantom{#1}}^{*}{\hskip -2pt #1}} 
\newcommand{\tld}{\widetilde }
\newcommand{\ba}{\mathfrak{a}}
\DeclareMathOperator{\an}{{an}}

\DeclareMathOperator{\Supp}{{Supp}}
\DeclareMathOperator{\coherent}{{coh}}
\DeclareMathOperator{\quasicoherent}{{qcoh}}
\DeclareMathOperator{\Ext}{Ext}
\DeclareMathOperator{\Hom}{Hom}
\DeclareMathOperator{\sHom}{{\sH}om}
\DeclareMathOperator{\Spec}{{Spec}}
\DeclareMathOperator{\sn}{{sn}}
\DeclareMathOperator{\wn}{{wn}}
\DeclareMathOperator{\Char}{{char}}
\DeclareMathOperator{\Gr}{{Gr}}
\DeclareMathOperator{\red}{red}
\DeclareMathOperator{\Frac}{{Frac}}
\DeclareMathOperator{\height}{{ht}}
\newcommand{\tensor}{\otimes}
\newcommand{\hypertensor}{{\uuline \tensor}}
\begin{document}

\title{F-injective singularities are Du Bois}
\author{Karl Schwede}
\address{Department of Mathematics\\
University of Michigan\\
530 Church Street\\
Ann Arbor MI 48109}
\email{kschwede@umich.edu}
\begin{abstract}In this paper, we prove that singularities of $F$-injective type are Du Bois.  This extends the correspondence between
singularities associated to the minimal model program and singularities defined by the action of Frobenius in positive characteristic.
\end{abstract}
\thanks{The author was partially supported by an NSF postdoctoral fellowship and by RTG grant
number 0502170}
\maketitle

\section{Introduction and background}

The main result of this paper is the following theorem:
\vskip10pt
\hskip-12pt
{\textsc{Theorem \ref{TheoremFInjectiveImpliesDuBois}.}}
{\it
Suppose that $X$ is a reduced scheme of finite type over a field of characteristic zero.  If $X$ has dense $F$-injective type, then $X$
has Du Bois singularities.
}
\vskip10pt \hskip -12pt
A reduced scheme of finite type over a perfect field of characteristic $p > 0$ is called $F$-injective if for every point $Q \in X$ with stalk $R_Q = \O_{X,Q}$, the action of Frobenius on each local cohomology module $H^i_{Q R_Q}(R_Q)$ is injective.  Now suppose that $X$ is a scheme of finite type over a field of characteristic zero.  Roughly speaking, we say that $X$ has dense $F$-injective type if, in a family of characteristic $p$ models for $X$ (obtained by reduction to positive characteristic), an infinite set of those models are themselves $F$-injective.   See Definitions \ref{DefinitionFInjective} and \ref{FBlankDefinition} for more details.

The notion of Du Bois singularities, which has its origins in Hodge theory (see \cite{SteenbrinkCohomologicallyInsignificant} and \cite{DuBoisMain}), is somewhat harder to define; see Definition \ref{DefinitionDuBoisSingularities} and Theorem \ref{AlternateDuBoisCharacterization}.  However, Du Bois singularities have the following important property:  If $X$ is a reduced proper scheme over $\bC$ with Du Bois singularities, then the natural map $H^i(X^{\an}, \bC) \rightarrow H^i(X, \O_X)$ is surjective.  Note that if $X$ is smooth, this surjectivity follows from the degeneration of the Hodge De Rham spectral sequence. Furthermore, this surjectivity is an important topological condition which naturally appears in several contexts (for example, see \cite[Chapter 9]{KollarShafarevich} to see how it relates to Kodaira-type vanishing theorems).

We now briefly review some of the history that leads us to the result above.  The minimal model program in dimension three, concerned with finding ``minimal'' birational models of algebraic varieties, was one of
the major developments in Algebraic Geometry in the 1980s.  A key aspect of the progress made at that time was the
discovery that one had to deal with singular, instead of simply smooth, varieties.  Some of the types of singularities
associated with the minimal model program are terminal, canonical, log terminal, log canonical, and rational singularities; see
\cite{KollarMori}, \cite{KMM}, \cite{KollarSingularitiesOfPairs} and \cite{ReidYoungPersons}.

The action of Frobenius, or $p$th power map, on rings of prime characteristic has been an important tool for studying singularities
for many years and has had a number of interesting applications to commutative algebra since the 1970s; see for example
\cite{HochsterRobertsFrobeniusLocalCohomology}, \cite{HartshorneSpeiserLocalCohomologyInCharacteristicP}, and
\cite{PeskineSzpiroDimensionProjective}.  In
the early 1980s, Richard Fedder demonstrated that $F$-pure and $F$-injective singularities (which are defined by the action of
Frobenius, see Remark \ref{FPureRemark}) were connected to rational singularities
\cite{FedderFPureRat}.  This was interesting because the definition of rational singularities is very different from the purely
algebraic definitions of $F$-pure and $F$-injective singularities.

In the mid 1980s, Melvin Hochster and Craig Huneke introduced tight closure \cite{HochsterHunekeTC1}, a powerful new method of
commutative
algebra that also relied on the Frobenius action.  Associated with tight closure were other new types of singularities, particularly
$F$-regular and $F$-rational singularities, which are more restrictive classes than $F$-pure and $F$-injective singularities
respectively.  An excellent survey article on tight closure from a geometric point of view is \cite{SmithTightClosure}.

In the 1990s, the correspondence between singularities associated to the minimal model program with those coming from the action of
Frobenius became well established.  Karen Smith proved that singularities of open $F$-rational type are rational; see
\cite{SmithFRatImpliesRat}.  The fact that rational singularities are of open $F$-rational type was independently proved by Hara and also by Mehta and Srinivas; see \cite{HaraRatImpliesFRat} and \cite{MehtaSrinivasRatImpliesFRat} respectively.
This result also implies that log terminal $\bQ$-Gorenstein singularities are of dense
$F$-regular type; see \cite{HaraRatImpliesFRat}.  Furthermore, Watanabe established that normal $\bQ$-Gorenstein singularities of open $F$-regular type (respectively dense $F$-pure type) are log terminal (respectively log canonical).  Watanabe's results were generalized to the context of pairs with the help of Hara and published in
\cite{HaraWatanabeFRegFPure}.  Further generalizations were made by Takagi; see \cite{TakagiInversion} and \cite{TakagiPLTAdjoint}.  It is still an open question
whether log canonical singularities are of dense $F$-pure type (although it is known in certain special cases; see
\cite{HaraDimensionTwo}, \cite{MehtaSrinivasFPureSurface}, \cite{SrinivasFPureType}, and \cite{TakagiWatanabeFPureThresh}).  Also see
\cite{HaraInterpretation}, \cite{SmithMultiplierTestIdeals}, and \cite{HaraYoshidaGeneralizationOfTightClosure} for additional
discussion of related notions.

$F$-injective singularities fit naturally into the lower right corner of the right square of the diagram below.  It is also known that Gorenstein
$F$-injective singularities are $F$-pure; see \cite{FedderFPureRat}.  On the other hand, Du Bois
singularities, historically connected to more analytic methods, (conjecturally) naturally fill
the lower right hand corner of the left square; see \cite{SteenbrinkMixed} and \cite{IshiiIsolatedQGorenstein}.  In
particular, rational singularities are known to be Du Bois, see \cite{KovacsDuBoisLC1} and \cite{SaitoMixedHodge}, and it was conjectured by Koll\'ar, see \cite[1.13]{KollarFlipsAndAbundance}, that log canonical singularities are Du Bois.  The best progress towards this conjecture can be found in
\cite{IshiiIsolatedQGorenstein}, \cite{KovacsDuBoisLC1}, \cite{KovacsDuBoisLC2}, and \cite{SchwedeEasyCharacterization}.  In \cite{KovacsDuBoisLC1}, it is also shown that normal quasi-Gorenstein Du Bois singularities are log canonical (also see \cite{DohertySingularitiesOfGenericProjectionHypersurfaces}).

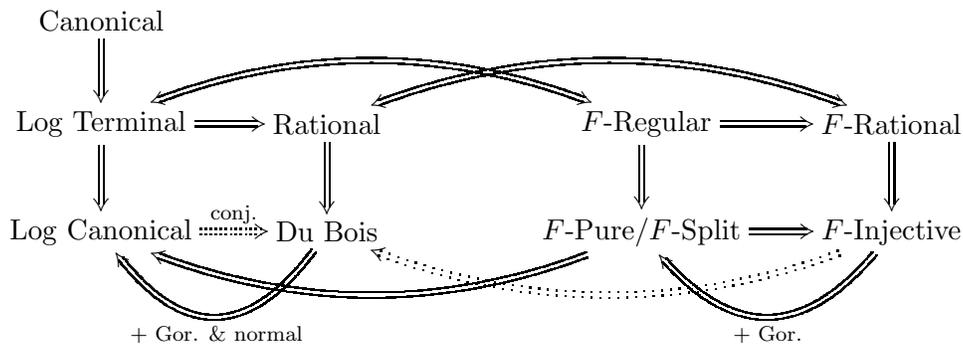
\[
\small
\xymatrix{
\text{Canonical} \ar@{=>}[d] & & & \\
\text{Log Terminal} \ar@/^2pc/@{<=>}[rrr] \ar@{=>}[r] \ar@{=>}[d]& \text{Rational} \ar@{=>}[d] \ar@/^2pc/@{<=>}[rrr] & & \text{
$F$-Regular} \ar@{=>}[r] \ar@{=>}[d] & \text{$F$-Rational} \ar@{=>}[d]\\
\text{Log Canonical} \ar@2{.>}[r]^-{\text{conj.}} \ar@/_2pc/@{<=}[rrr] \ar@/_2.7pc/@{<=}[r]_-{\text{+ Gor. \& normal}} & \text{Du Bois}
\ar@/_2pc/@2{<.}[rrr] & & \text{$F$-Pure/$F$-Split}
\ar@{=>}[r] \ar@/_2.7pc/@{<=}[r]_-{\text{+ Gor.}} & \text{$F$-Injective}\\
& & & & \\
}
\]

We propose a correspondence between $F$-injective and Du Bois singularities.  As mentioned before, the main goal of this paper is to
prove that singularities of dense $F$-injective type are Du Bois.  In the process of proving this result, we will show that
seminormality partially characterizes  Du Bois singularities and that weak normality partially characterizes $F$-injective
singularities; see Lemma \ref{SeminormalDuBoisComplex} and Theorem \ref{FInjectiveImpliesWeaklyNormal}.
\\
\\
{\it Acknowledgements:}
\\
The results that appear in this paper originally appeared in my doctoral dissertation at the University of Washington, which was
directed by S\'andor Kov\'acs.  I would like to thank Nobuo Hara for pointing out to me why $F$-injectivity localizes, Neil Epstein for
introducing me to the algebraic side of seminormality, Davis Doherty for several valuable discussions, and Karen Smith for comments on an earlier draft of this paper.  I would also like to thank the referee for many useful suggestions and for pointing out several typos in an earlier draft.

\section{Preliminaries}

All rings will be assumed to be commutative, noetherian and essentially of finite type over a field or $\bZ$.  In particular, all rings will be assumed to be excellent.  All schemes will be
assumed to be separated and noetherian.

We begin with a definition of a log resolution.
Let $\pi : \tld X \rightarrow X$ be a morphism of reduced schemes of finite type over a field $k$.  Suppose that $\ba$ is a sheaf of
ideals on $X$.  We say that $\pi$ is a \emph{log resolution} of $\ba$ if it satisfies the following four properties.
\begin{itemize}
\item[(i)]  $\pi$ is birational and proper (typically it is chosen to be projective).
\item[(ii)]  $\tld X$ is smooth over $k$
\item[(iii)]  $\ba \O_{\tld X} = \O_{\tld X}(-G)$ is an invertible sheaf corresponding to a divisor, $-G$.
\item[(iv)]  If $E$ is the exceptional set of $\pi$, then $\Supp(G) \cup E$ has simple normal crossings.
\end{itemize}

We say that $\pi$ is a \emph{strong log resolution} if it is a log resolution and $\pi$ is an isomorphism outside of the subscheme
$V(\ba)$ defined by $\ba$.  Log resolutions always exist if the characteristic of $k$ (the underlying field), is zero.  Strong log
resolutions always exist if $X$ is smooth over a field $k$ of characteristic zero;  see \cite{HironakaResolution}.  The reader may also
wish to consult \cite{BierstoneMilman}, \cite{BravoEncinasVillmayorSimplified}, or \cite{WlodarczykResolution} for an algorithmic
approach to constructing log resolutions.

We now review the forms of duality we will need.  We first fix some notation.  Suppose $X$ is a scheme.  The
notation $D^b(X)$ will denote the derived category of bounded complexes of $\O_X$-modules, $D^+(X)$ will correspond to bounded below
complexes and $D^-(X)$ will correspond to bounded above complexes.  The notation $D_{\coherent}(X)$ will denote the derived category of
$\O_X$-modules with coherent cohomology and $D_{\quasicoherent}(X)$ will denote the derived category of $\O_X$-modules with
quasi-coherent cohomology.  The various combinations $D^b_{\coherent}(X)$, $D^+_{\quasicoherent}(X)$, etc. will denote the obvious
derived categories, (ie. bounded complexes with coherent cohomology, bounded below complexes with quasi-coherent cohomology, etc.).  If
$F^{\mydot}$ and $G^{\mydot}$ are complexes, we will write $F^{\mydot} \qis G^{\mydot}$ if $F^{\mydot}$ and $G^{\mydot}$ are
quasi-isomorphic.  See \cite{HartshorneResidues} for precise definitions of these notions.

Recall that a \emph{dualizing complex} $\omega_X^{\mydot}$ for $X$ is an object in $D^+_{\coherent}(X)$ of finite injective dimension
such that the natural map
\[
F^{\mydot} \rightarrow \myR \sHom^{\mydot}_{\O_X}(\myR \sHom^{\mydot}_{\O_X}(F^{\mydot}, \omega_X^{\mydot}), \omega_X^{\mydot})
\]
is an isomorphism for all $F^{\mydot} \in D^+_{\coherent}(X)$; see \cite[V.2]{HartshorneResidues}.  Dualizing complexes exist for
schemes essentially of finite type over a field or $\bZ$, and are unique up to shifting and tensoring with an invertible sheaf
\cite[V.3.1]{HartshorneResidues}.  Note that this implies that dualizing complexes exist for all schemes considered in this paper.  Analogous definitions can be made for rings, and we will freely switch between affine schemes and
rings.

Often, we will be working with local rings $(R, m)$, in which case tensoring with an invertible sheaf is uninteresting.  However we can
specify the shift to a certain extent.  If $k = R/m$ is the residue field then there is a $d$ such that
\[
\Ext^i(k, \omega_R^{\mydot}) = 0 \text{ for $i \neq d$, and } \Ext^d(k, \omega_R^{\mydot}) = k, \text{
\cite[V.3.4]{HartshorneResidues}}.
\]
A dualizing complex is called \emph{normalized} if the $d$ above is equal to zero.

We are now in a position to state the duality theorems we will use.  First, we state local duality for complexes.

\begin{theorem}\cite[V.6.2]{HartshorneResidues}, \cite[2.4]{LipmanLocalCohomologyAndDuality}
\label{LocalDualityForComplexes}
Let $(A,m)$ be a local ring and of dimension $d$ and $C^{\mydot} \in D^+_{\coherent}(A)$.  Suppose $\omega_A^{\mydot}$ is a normalized
dualizing complex for $A$ and suppose $I$ is an injective hull of the residue field $k = A/m$.  Then the natural morphism of functors
\[
\myR \Gamma_m(C^{\mydot}) \rightarrow \myR \Hom^{\mydot}_A(\myR \Hom^{\mydot}_A(C^{\mydot}, \omega_A^{\mydot}), I)
\]
is an isomorphism.
\end{theorem}

Second, we state \emph{Grothendieck duality} for proper morphisms.

\begin{theorem}\cite[III.11.1, VII.3.4]{HartshorneResidues}, \cite{ConradGDualityAndBaseChange}
\label{GrothendieckDuality}
Let $f : X \rightarrow Y$ be a proper morphism of noetherian schemes of finite dimension.  Suppose $\sF^{\mydot} \in
D^{-}_{\quasicoherent}(X)$ and $\sG^{\mydot} \in D^{+}_{\coherent}(Y)$.  Then the duality morphism
\[
\myR f_* \myR \sHom_{\O_X}^{\mydot}(\sF^{\mydot}, f^{!} \sG^{\mydot}) \rightarrow \myR \sHom_{\O_Y}^{\mydot}(\myR f_* \sF^{\mydot},
\sG^{\mydot}),
\]
is an isomorphism.
\end{theorem}

\begin{remark}
The case we will usually consider is when $\sG^{\mydot}$ is a dualizing complex for $Y$ so that $f^{!}(\omega_Y^{\mydot}) =
\omega_X^{\mydot}$, giving us the following form of duality
\[
\myR f_* \myR \sHom_{\O_X}^{\mydot}(\sF^{\mydot}, \omega_X^{\mydot}) \qis \myR \sHom_{\O_Y}^{\mydot}(\myR f_* \sF^{\mydot},
\omega_Y^{\mydot}).
\]
In many cases $\sF^{\mydot}$ will be replaced by a module (viewed as a complex in degree zero).  One should take particular note that
when $f$ corresponds to a finite map of rings $A \rightarrow B$, then $f^!$ has a particularly nice interpretation.  Specifically,
$f^!(\blank) = \myR \Hom^{\mydot}_A(B, \blank)$; see \cite[III.6, VI]{HartshorneResidues}.
\end{remark}

\section{Seminormality}

In this section we lay out the basic definitions of seminormality and weak normality.  Recall that we have assumed that all rings are excellent.  

\begin{definition} \cite{AndreottiBombieri}, \cite{GrecoTraversoSeminormal}, \cite{SwanSeminormality}
\label{Subintegral}
A finite integral extension of reduced rings $i : A \subset B$ is said to be \emph{subintegral} (respectively \emph{weakly
subintegral}) if
\begin{itemize}
\item[(i)]  it induces a bijection on the prime spectra, and
\item[(ii)]  for every prime $P \in \Spec B$, the induced map on the residue fields, $k(i^{-1}(P)) \rightarrow k(P)$, is an isomorphism
(respectively, is a purely inseparable extension of fields).
\end{itemize}
\end{definition}

\begin{remark}
A subintegral extension of rings has also been called a quasi-isomorphism; see for example \cite{GrecoTraversoSeminormal}.
\end{remark}

\begin{remark}
Condition (ii) is unnecessary in the case of extensions of rings of finite type over an algebraically closed field of characteristic
zero.
\end{remark}

\begin{definition} \cite[1.2]{GrecoTraversoSeminormal}, \cite[2.2]{SwanSeminormality}
\label{SeminormalityExtensionDefinition}
Let $A \subset B$ be a finite extension of reduced rings.  Define $\SIClo{A}{B}$ to be the (unique) largest subextension of $A$ in $B$ such
that $A \subset \SIClo{A}{B}$ is subintegral.  This is called the \emph{seminormalization of $A$ inside $B$}.  $A$ is said to be
\emph{seminormal in $B$} if $A = \SIClo{A}{B}$.
\end{definition}

\begin{definition} \cite{YanagiharaWeaklyNormal}, \cite[1.1]{ReidRobertsSinghWeak}
\label{WeakNormalityExtensionDefinition}
Let $A \subset B$ be a finite extension of reduced rings.  Define $\WSIClo{A}{B}$ to be the (unique) largest subextension of $A$ in $B$ such
that $A \subset \WSIClo{A}{B}$ is weakly subintegral.  This is called the \emph{weak normalization of $A$ inside $B$}.  $A$ is said to be
\emph{weakly normal in $B$} if $A = \WSIClo{A}{B}$.
\end{definition}

\begin{definition}
\label{SemiWeakNormalityDefinition}
A reduced ring $A$ is said to be \emph{seminormal} (respectively \emph{weakly normal}) if it is seminormal (respectively weakly normal)
in its integral closure $\overline A$ (in its total field of fractions).  Its \emph{seminormalization} is $\SIClo{A}{\overline A}$ and
will be denoted by $\SN{A}$, respectively its \emph{weak normalization} is $\WSIClo{A}{\overline A}$ which will be denoted by $\WN{A}$.
If $X = \Spec A$ is a scheme, $X^{\sn}$ will be used to denote the scheme $\Spec \SIClo{A}{\overline A}$ and $X^{\wn}$ will be used to
denote the scheme $\Spec \WSIClo{A}{\overline A}$.  A scheme is said to be \emph{seminormal} (respectively, \emph{weakly normal}), if all
of its local rings are seminormal (respectively, weakly normal).
\end{definition}

\begin{remark}
Note the following set of implications.
\[
\xymatrix{
\text{Normal} \ar@{=>}[r] & \text{Weakly Normal} \ar@{=>}[r] & \text{Seminormal}
}
\]
Consider the following examples:
\begin{itemize}
\item[(i)] The union of two axes in $\bA^2$, $\Spec k[x,y]/(xy)$, is both weakly normal and seminormal, but not normal (an irreducible node is seminormal as well).
\item[(ii)]  The union of three lines through the origin in $\bA^2$, $\Spec k[x,y]/(xy(x-y))$, is neither seminormal nor weakly normal.
\item[(iii)] The union of three axes in $\bA^3$, $\Spec k[x,y,z]/(x,y) \cap (y,z) \cap (x,z)$, is both seminormal and weakly normal.  In fact, it is isomorphic to the seminormalization of (ii).
\item[(iv)]  The pinch point $\Spec k[a,b,c]/(a^2 b - c^2) \cong \Spec k[x^2, y, xy]$ is both seminormal and weakly normal as long as the characteristic of $k$ is not equal to two.  In the case that $\Char k = 2$, then the pinch point is seminormal but not weakly normal.  Notice that if $\Char k = 2$ then the inclusion $k[x^2, y, xy] \subset k[x,y]$ induces a bijection on spectra.  Furthermore the induced maps on residue fields are isomorphisms at all closed points.  However, at the generic point of the singular locus $P = (y, xy)$, the induced extension of residue fields is purely inseparable.  This proves that it is not weakly normal.
\end{itemize}
\end{remark}

\begin{remark}
Seminormality and weak normality localize; see \cite{GrecoTraversoSeminormal} and \cite{YanagiharaWeaklyNormal}.  In particular, a scheme is seminormal if and only if it
has an affine cover by the spectrums of seminormal rings.
\end{remark}

There are other characterizations of weakly normal and seminormal which are of a more algebraic nature, and are often very useful.

\begin{proposition} \cite[1.4]{VitulliSeminormal}
\label{SemiNormalEquivalentConditions}
Let $A \subset B$ be a finite integral extension of reduced rings; the following are then equivalent:
\begin{itemize}
\item[(i)]  $A$ is seminormal in $B$
\item[(ii)]  For a fixed pair of relatively prime integers $e > f > 1$, $A$ contains each element $b \in B$ such that $b^e, b^f \in A$.
(also see \cite{HamannSeminormal} and \cite{SwanSeminormality} for the case where $e = 2$, $f = 3$).
\end{itemize}
\end{proposition}

\begin{proposition}  \cite[4.3, 6.8]{ReidRobertsSinghWeak}
\label{WeaklyNormalCondition}
Let $A \subset B$ be a finite integral extension of reduced rings where $A$ contains $\bF_p$ for some prime $p$; the following are then
equivalent:
\begin{itemize}
\item[(i)]  $A$ is weakly normal in $B$.
\item[(ii)]  If $b \in B$ and $b^p \in A$ then $b \in A$.
\end{itemize}
\end{proposition}

The property of seminormality is also intrinsic in the following sense:

\begin{theorem} \cite{SwanSeminormality}
\label{SeminormalityIntrinsic}
A reduced ring $A$ is seminormal if and only if for every pair of elements $b, c \in A$ such that $b^3 = c^2$, there exists a unique $a
\in A$ such that $a^2 = b$ and $a^3 = c$.
\end{theorem}

\begin{remark}
There is no equivalently intrinsic characterization of weak normality.  See \cite{YanagiharaWeaklyNormal} for further discussion.
\end{remark}

\begin{corollary}
\label{GlobalSectionsOfSeminormality}
Suppose $X$ is a seminormal scheme.  Then for any open $U \subset X$, $\Gamma(U, \O_X)$ is a seminormal ring.
\end{corollary}
\begin{proof}
Without loss of generality we may assume that $U = X$.  Note that $\Gamma(X, \O_X)$ is necessarily a reduced ring since $X$ is a
reduced scheme.  Choose an affine cover of $X$ by the spectrums of seminormal rings, $A_i$.  Suppose there exist elements $b, c \in
\Gamma(X, \O_X)$ such that $b^3 = c^2$.  Then these elements have images in each $A_i$, which we will call $b_i$ and $c_i$
respectively, and note that these images still satisfy the relation $b_i^3 = c_i^2$.  Thus, in each $A_i$ there exists a unique $a_i$
such that $a_i^2 = b_i$ and $a_i^3 = c_i$.  The uniqueness of these $a_i$ guarantee that they will glue together to an element $a$ of
$\Gamma(X, \O_X)$.  But then, we must also have $a^2 = b$ and $a^3 = c$ and because $\Gamma(X, \O_X)$ is reduced, $a$ is the unique
element satisfying this property; see \cite{SwanSeminormality}.
\end{proof}

\section{$F$-injective singularities}

Let $R$ be a ring containing a field of characteristic $p$.  By the \emph{action of Frobenius} we mean the $p$th power map, $F : R
\rightarrow R$, which sends $x$ to $x^p$.  If $R$ is reduced, this map may be reinterpreted as $R^p \rightarrow R$ or $R \rightarrow
R^{1/p}$.  We will use ${}^1 R$ to denote $R$ viewed as an $R$-module via the action of Frobenius.  Recall that a ring is said to be
\emph{$F$-finite} if ${}^1 R$ is a finite $R$-module.  Note that a ring of finite type over a finite field is always $F$-finite.

\begin{definition} \cite{FedderFPureRat}
\label{DefinitionFInjective}
Suppose $(R,m)$ is a reduced local ring containing a field of characteristic $p$.  We say that $R$ is \emph{$F$-injective} if the map
$F : H^i_m(R) \rightarrow H^i_m({}^1R)$, induced by the action of Frobenius, injects for all $i$.
\end{definition}

\begin{remark}
\label{FPureRemark}
There is another closely related class of singularities called $F$-pure singularities.  A reduced local ring $R$ is said to be \emph{$F$-pure} if the Frobenius map $F: R \rightarrow {}^1 R$ is a pure map of rings.  If $R$ is $F$-finite, this is equivalent to the map splitting as a map of $R$-modules. By applying $\myR \Gamma_m$ to such a splitting, it is obvious that $F$-finite $F$-pure local rings are $F$-injective.  See \cite{FedderFPureRat} for additional discussion of $F$-pure and $F$-injective rings.
\end{remark}

The property of $F$-injectivity localizes, as proven in the proposition below.  I would like to thank Nobuo Hara for suggesting the
proof used in this paper.  This fact certainly was known previously, but I do not know of a reference.

\begin{proposition}
\label{FInjectiveLocalizes}
Suppose $(S,m)$ is an $F$-finite $F$-injective local ring with a dualizing complex and $P \in \Spec S$.  Then $(S_P, P)$ is an $F$-injective local ring.
\end{proposition}
\begin{proof}
\numberwithin{equation}{theorem}
Let $X = \Spec S$ and let $I$ be an injective hull of the residue field $S/m$.  Frobenius induces a map $\O_X \rightarrow F_* \O_X$.
Note $F_* \O_X$ can be identified with $\myR F_* \O_X$ since $F$ is an affine (in fact finite) map.   Since $(S,m)$ is $F$-injective,
we have
\[
H^i_m(X, \O_X) \rightarrow \bH^i_m(X, \myR F_* \O_X)
\]
injective for all $i$.  Therefore, by Matlis duality, we see that
\[
\Hom_S(\bH^i_m(X, \myR F_* \O_X), I) \rightarrow \Hom_S(H^i_m(X, \O_X), I)
\]
is surjective for each $i$, since $\Hom_S(\blank, I)$ is exact.

Finally we apply local duality to see that
\begin{align*}
\Hom_S(\Hom_S(h^{-i}(\myR \sHom_{\O_X}(\myR F_* \O_X, \omega_X^{\mydot}), I), I)) & \rightarrow \\
\Hom_S(\Hom_S(h^{-i}(\myR \sHom_{\O_X}(\O_X, \omega_X^{\mydot}), I), I))
\end{align*}
is surjective for each $i$ (note that we have pulled the cohomology functor, $h^i$ which became $h^{-i}$, through $\Hom_S(\blank, I)$, which we can do since $\Hom_S(\blank, I)$ is an exact functor).  We have also abused notation by freely switching between a sheaf on an affine scheme and the associated module.
Note that for a finite module $M$ over a local ring $R_P$,
\[
\Hom_{\hat{R_P}}^{\mydot}(\Hom_{R_P}^{\mydot}(M, I), I)
\]
is naturally identified with $\hat{M}$ where $\hat{\blank}$ is completion, \cite[3.5.4]{BrunsHerzog}, \cite[Exercise
3.2.14(b)]{BrunsHerzog} or see \cite[10.2.19]{BrodmannSharpLocalCohomology}.  Since completion of a local ring does not send finite
non-zero modules to zero (or, more generally, since completion is faithfully flat \cite[Section 24]{MatsumuraCommutativeAlgebra}), this implies that
\begin{equation}
\label{DualizedFInjection}
h^{-i}(\myR \sHom_{\O_X}^{\mydot}(\myR F_* \O_X, \omega_X^{\mydot})) \rightarrow h^{-i}(\myR \sHom^{\mydot}_{\O_X}(\O_X,
\omega_X^{\mydot}))
\end{equation}
is surjective for every $i$ (see the proof of Proposition \ref{GeneralizedKovacsSurjectivity} for a similar argument).  Now, when we localize the map
\ref{DualizedFInjection} at $P$, it remains surjective (since $R$ is $F$-finite).  Since all the relevant objects and functors behave well with respect to localization, we obtain a surjection
\[
h^{-i}(\myR \sHom_{\O_{X_P}}^{\mydot}(\myR F_* \O_{X_P}, \omega_{X_P}^{\mydot})) \rightarrow h^{-i}(\myR \sHom^{\mydot}_{\O_{X_P}}(\O_{X_P},
\omega_{X_P}^{\mydot})).
\]
Note that we are suppressing a shift, determined by the height of $P$, which normalizes the dualizing complex $\omega_{X_P}^{\mydot}$.  We then apply local duality again to complete the proof.
\end{proof}

\begin{definition}
Let $R$ be any ring containing a field of characteristic $p$ and suppose that $R$ is $F$-finite and has a dualizing complex.  Then we say that $R$ is \emph{$F$-injective} if $(R_P, P)$ is $F$-injective for every $P \in \Spec R$.
\end{definition}

\begin{remark}
If $R$ is $F$-finite and has a dualizing complex, note that $R$ is $F$-injective if and only if $(R_m, m)$ is $F$-injective for every maximal ideal in $\Spec R$.
Fedder defined an arbitrary ring to be $F$-injective if $R$ localized at every maximal ideal is $F$-injective.
\end{remark}

We also need to show that $F$-injective singularities are seminormal (we will actually show more: we will prove that they are weakly
normal).  It is known that $F$-pure singularities are seminormal; see \cite[5.31]{HochsterRobertsFrobeniusLocalCohomology}, also see \cite{GibsonSeminormalityInDimensionOne}.

\begin{lemma}
\label{LemmaFInjectiveImpliesSeminormal}
Suppose that $(R, m)$ is a reduced local ring of characteristic $p$, $X = \Spec R$ and that $X \setminus m$ is weakly normal.  Then $X$ is
weakly normal if and only if the action of Frobenius is injective on $H^1_m(R)$.
\end{lemma}
\begin{proof}
We assume that the dimension of $R$ is greater than $0$ since the zero-dimensional case is trivial.  Embed $R$ in its weak
normalization $R \subset \WN{R}$ (which is of course an isomorphism outside of $m$).  We have the following diagram of $R$-modules.
\[
\xymatrix{
0 \ar[r] & R \ar@{^{(}->}[r] \ar@{^{(}->}[d] & \Gamma(X \setminus m, \O_{X-m}) \ar[d]^{\cong} \ar@{->>}[r] & H^1_m(R) \ar[r] \ar[d] & 0 \\
0 \ar[r] & \WN{R} \ar@{^{(}->}[r] & \Gamma(X^{\wn} \setminus m, \O_{X^{\wn}-m}) \ar@{->>}[r] & H^1_m(\WN{R}) \ar[r] & 0 \\
}
\]
The left horizontal maps are injective because $R$ and $\WN{R}$ are reduced.  One can check that Frobenius is compatible with all of
these maps.  Now, $R$ is weakly normal if and only if $R$ is weakly normal in $\WN{R}$ if and only if every $r \in \WN{R}$ with $r^p \in
R$ also satisfies $r \in R$ by Proposition \ref{WeaklyNormalCondition}.

First assume that the action of Frobenius is injective on $H^1_m(R)$.  So suppose that there is such an $r \in \WN{R}$ with $r^p \in R$.
Then $r$ has an image in $\Gamma(X \setminus m, \O_{X-m})$ and therefore an image in $H^1_m(R)$.  But $r^p$ has a zero image in $H^1_m(R)$,
which means $r$ has zero image in $H^1_m(R)$, which guarantees that $r \in R$ as desired.

Conversely, suppose that $R$ is weakly normal.  Let $r \in \Gamma(X \setminus m, \O_{X-m})$ be an element such that the action of Frobenius
annihilates its image $\overline r$ in $H^1_m(R)$.  Since $r \in \Gamma(X \setminus m, \O_{X-m})$ we identify $r$ with a unique element of the
total field of fractions of $R$.  On the other hand, $r^p \in R$ so $r \in \WN{R} = R$.  Thus we obtain that $r \in R$ and so
$\overline r$ is zero as desired.
\end{proof}

\begin{theorem}
\label{FInjectiveImpliesWeaklyNormal}
Let $R$ be a reduced $F$-finite rin with a dualizing complex.  If $R$ is $F$-injective then $R$ is weakly normal (and thus in particular seminormal).
\end{theorem}
\begin{proof}
A ring is weakly normal if and only if all its localizations at prime ideals are weakly normal \cite[6.8]{ReidRobertsSinghWeak}.  If
$R$ is not weakly normal, choose a prime $P \in \Spec R$  of minimal height with respect to the condition that $R_P$ is not weakly
normal.  Apply Lemma \ref{LemmaFInjectiveImpliesSeminormal} to get a contradiction.
\end{proof}

$F$-injective singularities also glue together nicely.  Compare this result with \cite[3.8, 4.10]{DuBoisMain}.

\begin{proposition}
\label{GluingFInjectiveSingularities}
Let $X$ be a $F$-finite reduced scheme over a field of characteristic $p$.  Suppose $X$ is the union of two closed Cohen-Macaulay
subschemes, $Y_1$ and $Y_2$, of the same dimension.  If $Y_1 \cap Y_2$ is reduced, and $Y_1$, $Y_2$ and $Y_1 \cap Y_2$ all have
$F$-injective singularities, then $X$ has $F$-injective singularities.
\end{proposition}
\begin{proof}
Gluing in this manner behaves well with respect to localization, so we may assume that $X$ is the spectrum of a local ring $(R,m)$ and
that $I_1, I_2 \subset R$ are the defining ideals of $Y_1$ and $Y_2$.  Note since $I_1 \cap I_2 = 0$, we have a short exact sequence
\begin{equation}
\label{fundamentalGluingSES}
0 \rightarrow R \rightarrow R/I_1 \oplus R/I_2 \rightarrow R/(I_1 + I_2) \rightarrow 0 .
\end{equation}
Now we apply $\myR \Gamma_m$ to obtain a long exact sequence
\[
\xymatrix{
\ldots \ar[r] & H^i_m(R) \ar[r] & H^i_m(R/I_1) \oplus H^i_m(R/I_2) \ar[r] & H^i_m(R/(I_1 + I_2)) \ar[r]^-{\delta_i} & \\
\ar[r] & H^{i+1}_m(R) \ar[r] & H^{i+1}_m(R/I_1) \oplus H^{i+1}_m(R/I_2) \ar[r] & \ldots .\\
}
\]
We wish to consider the action of Frobenius on each term of this long exact sequence.  Note that the action is compatible since the
action of Frobenius on each individual ring in the short exact sequence \ref{fundamentalGluingSES} is compatible with the maps of that
sequence.

Recall that $R/I_1$, $R/I_2$ and $R$ all have the same (pure) dimension $d$.  Then since $R/I_1$ and $R/I_2$ are Cohen-Macaulay we see
that the $\delta_i$ induce isomorphisms between $H^i_m(R)$ and $H^{i-1}_m(R/(I_1 + I_2))$ for $i < d$; thus we need only consider the
case of $i = d$.  We have the following exact sequence to consider
\[
\xymatrix{
0 \ar[r] & H^{d-1}_m(R/(I_1 + I_2)) \ar[r]^-{\delta_{d-1}} & H^d_m(R) \ar[r]^-{\phi_d} & H^d_m(R/I_1) \oplus H^d_m(R/I_1) .\\
}
\]
So suppose that $x \in H^d_m(R)$ and $x$ is annihilated by the action of Frobenius.  Then $\phi_d(x) = 0$ and so there exists a $y \in
H^{d-1}_m(R/(I_1 + I_2))$ such that $\delta_{d-1}(y) = x$.  But the action of Frobenius is injective on $H^{d-1}_m(R/(I_1 + I_2))$ and
$\delta_{d-1}$ injects, proving that $x = 0$ as desired.
\end{proof}

\section{Du Bois singularities}

Historically, Du Bois singularities were defined using Hodge-theoretic methods.  In particular, for each reduced separated scheme of
finite type over $\bC$ (or any field of characteristic zero) one can associate a filtered complex $(\FullDuBois{X}, F)$ (in the
filtered derived category of $X$ with differentials of orders $\leq 1$) corresponding to the De Rham complex for smooth varieties; see
\cite{DeligneHodgeIII}, \cite{DuBoisMain}, \cite{SteenbrinkgVanishing} and \cite{GNPP} for constructions and see \cite[Chapter
12]{KollarShafarevich} and \cite{KovacsDuBoisLC1} for enumerations of basic properties.  Each graded piece of this complex, $\Gr_F^i
\FullDuBois{X}$ is an object in $D^b_{\coherent}(X)$ and, in the case that $X$ is smooth, is quasi-isomorphic to $\Omega_X^i[-i]$.  In
particular the zeroth graded piece of this complex $\DuBois{X} \cong \Gr_F^0 \FullDuBois{X}$ is a complex related to $\Omega^0_X$,
which is the structure sheaf of $X$. Furthermore, there is a natural map $\O_X \rightarrow \DuBois{X}$ in $D^b_{\coherent}(X)$.

\begin{definition}
\label{DefinitionDuBoisSingularities}
$X$ is said to have \emph{Du Bois singularities} if the natural map $\O_X \rightarrow \DuBois{X}$ is a quasi-isomorphism.
\end{definition}

Du Bois singularities behave well with respect to general hyperplane sections, specifically,

\begin{proposition} \cite[V, 1]{GNPP}, \cite[12.6.2]{KollarShafarevich}
\label{DuBoisHyperplane}
Suppose that $X$ is a reduced separated scheme of finite type over $\bC$.  Suppose that $D$ is a sufficiently general member of a base point free linear system.
Then we have the following:
\[
\DuBois{D} \qis \O_D \tensor \DuBois{X} \text{, and } h^i(\DuBois{D}) \cong \O_D \tensor h^i(\DuBois{X}).
\]
\end{proposition}

There is another characterization of the complex $\DuBois{X}$ which we will find to be very useful.

\begin{theorem} \cite{SchwedeEasyCharacterization}, \cite{KarlThesis}
\label{AlternateDuBoisCharacterization}
Let $X$ be a reduced separated scheme of finite type over a field of characteristic zero embedded in a variety $Y$.  Call $\sI_X$ the
ideal sheaf of $X$ in $Y$.
Assume $Y$ is smooth and let $\pi : \tld Y \rightarrow Y$ be a strong log resolution of $X$; that is, $\tld Y$ is smooth, the pre-image
of $X$ has simple normal crossings and $\pi$ is an isomorphism outside of $X$.
Let $E$ be the reduced total transform (pre-image) of $X$; (note that this is a divisor).  Then $\DuBois{X} \qis \myR \pi_* \O_E$.
\end{theorem}

\begin{remark}
The proof of this theorem is essentially an application of the octahedral axiom.  The main ideas were present in
\cite[7.7]{DuBoisDuality} and \cite[3.2]{KovacsDuBoisLC1}.  The quasi-isomorphism $\DuBois{X} \qis \myR \pi_* \O_E$ also identifies the
natural map $\O_X \rightarrow \DuBois{X}$ with the natural map $\O_X \rightarrow \myR \pi_* \O_E$.
\end{remark}

\begin{corollary}
With the notation from Theorem \ref{AlternateDuBoisCharacterization}, $X$ has Du Bois singularities if and only if the natural map $\O_X
\rightarrow \myR \pi_* \O_E$ is a quasi-isomorphism.
\end{corollary}

We relate Du Bois singularities to seminormality.  Du Bois knew that Du Bois singularities were seminormal, although he did not use the
language of seminormality; see \cite[4.8, 4.9]{DuBoisMain}.  However, even more is true.  In \cite[5.2]{SaitoMixedHodge}, it was shown that $h^0(\DuBois{X})$ is the structure sheaf of the seminormalization of $X$.  For completeness, an alternate proof which uses Theorem \ref{AlternateDuBoisCharacterization} is provided below.

\begin{lemma}\cite[5.2]{SaitoMixedHodge}
\label{SeminormalDuBoisComplex}
Suppose that $X$ is a reduced separated scheme of finite type over an algebraically closed field of characteristic zero.  Then
$h^0(\DuBois{X}) = \O_{X^{\sn}}$ where $\O_{X^{\sn}}$ is the structure sheaf of the seminormalization of $X$.
\end{lemma}
\begin{proof}
Without loss of generality we may assume that $X$ is affine.
We need only consider $\pi_* \O_E$ by Theorem \ref{AlternateDuBoisCharacterization}.  By Corollary \ref{GlobalSectionsOfSeminormality} $\pi_* \O_E$ is a sheaf of seminormal rings.  Now let $X' = \Spec (\pi_* \O_E)$ and consider the factorization
\[
E \rightarrow X' \rightarrow X.
\]
Note $E \rightarrow X'$ must be surjective since it is dominant by construction and is proper by \cite[II.4.8(e)]{Hartshorne}.  Since
the composition has connected fibers, so must $\rho : X' \rightarrow X$.  On the other hand, $\rho$ is a finite map since $\pi$ was
proper.  Therefore $\rho$ is a bijection on points.  Because these maps and schemes are of finite type over an algebraically closed
field of characteristic zero, we see that $\Gamma(X, \O_X) \rightarrow \Gamma(X', \O_{X'})$ is a subintegral extension of rings.  Since
$X'$ is seminormal, so is $\Gamma(X', \O_{X'})$, which completes the proof.
\end{proof}

\begin{corollary}
Suppose that $X$ is a reduced separated scheme of finite type over a field $k$ of characteristic zero.  Then $h^0(\DuBois{X}) = \O_{X}$
if and only if $X$ is seminormal.
\end{corollary}
\begin{proof}
The only thing to do here is to check statements of base change.  The condition that $X$ is seminormal is invariant under separable
(faithfully flat) base change of $k$ by \cite[5.7]{GrecoTraversoSeminormal}.  Likewise the condition that
\[
\O_{X} \qis \myR \pi_* \O_E
\]
is also invariant under base change of $k$.  This completes the proof.
\end{proof}

\begin{remark}
Compare the previous seminormality results with \cite{AmbroSeminormalLocus}.
\end{remark}

\begin{remark}
Proposition \ref{DuBoisHyperplane}, Lemma \ref{SeminormalDuBoisComplex} and standard results about hyperplane sections and log resolutions
also allow one to easily re-obtain certain results on general hyperplanes sections of seminormal varieties; see
\cite{VitulliHyperplaneSections} and \cite{CuminoGrecoManaresiBertiniAndWeakNormality}.  Related statements about general hyperplane sections of the seminormalization of a variety also
easily follow.
\end{remark}

We need to prove a certain injectivity into local cohomology in order to prove that singularities of dense $F$-injective type are Du
Bois.  We generalize the following lemma of Kov\'acs to the case of a non-closed point.  Kov\'acs' lemma often plays the same role for
Du Bois singularities that Grauert-Riemenschneider vanishing, \cite{GRVanishing}, plays for rational singularities.

\begin{lemma}  \cite[1.4]{KovacsDuBoisLC2}
\label{KovacsKeyLemma}
Let $X$ be a complex scheme with a finite set of closed points, $P$, such that $X \setminus P$
has only Du Bois singularities. Then
\[
H^i_P(X, \O_X) \rightarrow \bH^i_P(X, \DuBois{X})
\]
is surjective for all $i$.
\end{lemma}

Below is the slightly generalized version of this lemma.  The idea of the proof is to use local duality and apply general hyperplane
sections to reduce to Lemma \ref{KovacsKeyLemma}.

\begin{proposition}
\numberwithin{equation}{theorem}
\label{GeneralizedKovacsSurjectivity}
Suppose that $X = \Spec R$ is a reduced closed subscheme of $Y = \bA_{\bC}^n = \Spec \bC[x_1, \ldots, x_n]$.  Further suppose that $P
\in X = \Spec R$ is a prime such that $(\Spec R_P) \setminus P$ has Du Bois singularities (for example, this occurs if $P$ is a minimal prime
of the non-Du Bois locus).  Then
\[\xymatrix{H^i_P(R_P) \ar@{->>}[r] & \bH^i_P(X, \DuBois{X_P})}\]
is surjective for every $i$, or dually,
\[
\xymatrix{
h^{j} ((\myR \pi_* \omega_{E}^{\mydot})_P) \ar@{^{(}->}[r] & h^{j}(\omega_{X}^{\mydot})_P
}
\]
is injective for every $j$. Here $E$ is the pre-image of $X$ under a strong log resolution $\pi : \tld Y \rightarrow Y$ of $X$ as above and
$\blank_P$ is the localization at $P$.
\end{proposition}
\begin{proof}
As in Theorem \ref{AlternateDuBoisCharacterization}, we have $X$ in $Y$, where $Y$ is a smooth variety.  Perform a strong log resolution of
$X$ in $Y$ and let $E$ be the reduced pre-image.  Let $k = R_P/P R_P$ be the residue field of $R_P$ and let $I$ be an injective hull of
$k$.

Because the proof is rather long, we break it up into several steps.

\vskip 12pt
\hskip -12pt {\bf Step 1.}  \emph{We use local duality to reduce our question to a statement about dualizing complexes}.

We have a natural map $\O_X \rightarrow \myR \pi_* \O_E$ which we localize at $P$ (that is, take the product with $\Spec \O_{X_P}$ over
$X$) to obtain
\begin{equation}
\label{FirstMap}
\O_{X_P} \rightarrow \myR \pi_* \O_{E_P}.
\end{equation}
Now apply $\myR \Hom^{\mydot}_{\O_{X_P}}(\blank, \omega_{X_P}^{\mydot})$.  By Theorem \ref{GrothendieckDuality} (Grothendieck duality)  we see
that \ref{FirstMap} is dual to $\myR \pi_* \omega_{E_P}^{\mydot} \rightarrow \omega_{X_P}^{\mydot}$.  Apply $\myR
\Hom_{R_P}^{\mydot}(\blank, I)$; local duality then identifies this map with
\[
\myR \Gamma_P(X, \O_{X_P}) \rightarrow \myR \Gamma_P(X, \myR \pi_* \O_{E_P}).
\]
Note that since $\Hom_{R_P}(\blank, I)$ is exact, we abuse notation and associate it with $\myR \Hom_{R_P}^{\mydot}(\blank, I)$, its right derived functor.
If we apply $\myR \Hom^{\mydot}_{R_P}(\blank, I)$ to these complexes, as in the proof of Proposition \ref{FInjectiveLocalizes}, Matlis duality guarantees that each
\[
h^i(\myR \Gamma_P(X, \O_{X_P})) \rightarrow h^i(\myR \Gamma_P(X, \myR \pi_* \O_{E_P}))
\]
map is surjective if and only if each
\begin{equation}
\label{MapWeWantToInject}
\xymatrix{
\Hom_{R_P}(\Hom_{R_P}(h^{-i}(\myR \pi_* \omega_{E_P}^{\mydot}), I), I) \ar[r] & \Hom_{R_P}(\Hom_{R_P}(h^{-i}(\omega_{X_P}^{\mydot}),
I), I)  \\
}
\end{equation} is injective (since $\Hom_{R_P}(\blank, I)$ is exact).
By the same argument as the proof of Proposition \ref{FInjectiveLocalizes} we see that
\[
h^{-i}(\myR \pi_* \omega_{E_P}^{\mydot}) \rightarrow h^{-i}(\omega_{X_P}^{\mydot})
\]
injects if and only if the map above, \ref{MapWeWantToInject}, injects, and so we are reduced to proving this new injectivity.  In
other words, our desired result is equivalent to
\begin{equation}
\label{DualMapEquation}
h^{-i}(\myR \pi_* \omega_{E}^{\mydot}) \rightarrow h^{-i}(\omega_{X}^{\mydot})
\end{equation}
injecting at the stalk associated to $P$ (since $P$ is not maximal, there is a shift of complexes here that I am suppressing).  This conclude step 1.

\vskip 12pt \hskip -12pt{\bf Step 2.}  \emph{We explain the inductive setup and the choice of a general hyperplane}.

So suppose the map \ref{DualMapEquation} does not inject at the stalk associated to $P$.  We apply induction on the dimension of $\Spec (R/P)$ (the case
where $P$ has maximal height is simply Lemma \ref{KovacsKeyLemma}).

Choose a general hyperplane $H' \subset Y$.  Let $f'$ be the generator of the ideal of $H'$.  We will use the fact that $H'$ is general
in the following way.  The element $f'$ has nonzero, nonunit image $f \in R$, and geometrically speaking, will intersect $X = \Spec R$ in a lower
dimensional subscheme.  Furthermore we may view $f'$ as a global section of $E$ by the map $\O_X \rightarrow \pi_* \O_E$, and so we may
require that $f$ does not annihilate any elements of $\O_E$ on any open set of a fixed affine cover (since it also corresponds to a general member of a base point free linear system on $\tld Y$).  Note that $H' = \Spec \bC[x_1,
\ldots, x_n]/f'$ is smooth and $\pi$ induces a log resolution $\tld H' \rightarrow H'$ of $H = \Spec R/fR$ with reduced exceptional
divisor $E_H$, where $E_H$ corresponds to $\O_E / f \O_E$.  Similarly, we choose $f$ sufficiently generally so that it satisfies the
conclusion of Proposition \ref{DuBoisHyperplane}.  Finally we require that $f$ does not annihilate any element of $h^{-i}(\myR \pi_*
\omega_{E}^{\mydot})$, any element of $h^{-i}(\omega_X^{\mydot})$ or any elements in the kernel or cokernel of the natural maps between
these various modules.  This completes step 2.

\vskip 12pt \hskip-12pt{\bf Step 3.}  \emph{We use homological algebra and our general hyperplane to reduce the dimension of $R/P$ in order to apply our inductive hypothesis}.

Apply $\myR \Hom^{\mydot}_R(R/f, \blank) = \myR \Hom^{\mydot}_{\O_X}(\O_H, \blank)$ to $\myR \pi_* \omega_E^{\mydot} \rightarrow
\omega_{X}^{\mydot}$.  Because
\[
\xymatrix{R \ar[r]^{\times f} & R \ar[r] & R/f}
\]
is a (free) resolution of $R/f$, we see that $\myR \Hom^{\mydot}_R(R/f, M) \qis (M/fM)[-1] \qis M \tensor_R (R/f)[-1]$ for
any $R$-module $M$ with no elements annihilated by $f$.  Since we chose $f$ to be general, the hyper-cohomology spectral sequences computing $\myR
\Hom^{\mydot}_{\O_X}(\O_H, \blank)$ applied to the complexes $\myR \pi_* \omega_E^{\mydot}$ and $\omega_{X}^{\mydot}$ collapse (that is, have only one non-zero column).  Thus we
consider the maps
\begin{equation}
\label{firstpointEquation}
h^{-i}(\myR \pi_* \omega_E^{\mydot}) \tensor \O_H \rightarrow h^{-i}(\omega_{X}^{\mydot}) \tensor \O_H,
\end{equation}
coming from the map of the two spectral sequences.
Now note that $\myR \Hom^{\mydot}_{\O_X}(\O_H, \omega_{X}^{\mydot})$ is just $\omega_{H}^{\mydot}$ by the definition of $f^!$ for
finite maps.  Consider $\myR \Hom^{\mydot}_{\O_X}(\O_H, (\myR \pi_* \omega_E^{\mydot}) )$.  This can be viewed as
\[
\myR \Hom^{\mydot}_{\O_X}(\O_H, \myR \Hom^{\mydot}_{\O_X}(\myR \pi_* \O_E, \omega_X^{\mydot}))
\]
which is naturally isomorphic to
\begin{equation}
\label{midpointEquation}
\myR \Hom^{\mydot}_{\O_X}( (\O_H \hypertensor (\myR \pi_* \O_E)), \omega_X^{\mydot})
\end{equation}
by \cite[II.5.15]{HartshorneResidues} (here $\hypertensor$ denotes the left derived functor of $\tensor$ and the tensor product is over $\O_X$).
However,
\[
\O_H \hypertensor \myR \pi_* \O_E \qis \myR \pi_* \O_{E_H},
\]
also see Proposition \ref{DuBoisHyperplane}.  That quasi-isomorphism is also functorial in the sense that we have the following
commutative diagram.
\[
\xymatrix{
\O_H \hypertensor \O_X \ar[r]^-{\cong} \ar[d] & \O_H \ar[d]\\
\O_H \hypertensor \myR \pi_* \O_E \ar[r]^-{\cong} & \myR \pi_* \O_{E_H} \\
}
\]
The commutativity is particularly easy to see if one employs a \Cech resolution (made up of $\blank \tensor \O_H$-acyclic objects) to
compute $\myR \pi_* \O_E$.
By Grothendieck duality, and the description of $f^!$ for finite maps, we see that \ref{midpointEquation} can be re-written as
\[
\myR \pi_* \sHom^{\mydot}_{\O_E}(\O_{E_H}, \omega_E^{\mydot}) \qis \myR \pi_* \omega^{\mydot}_{E_H}.
\]
So we see that \ref{firstpointEquation} is compatibly identified with the natural map
\[
h^{-i+1}(\myR \pi_* \omega_{E_H}^{\mydot}) \rightarrow h^{-i+1}(\omega_{H}^{\mydot}).
\]
Therefore if the map we were originally interested in,
\[
h^{-i}(\myR \pi_* \omega_{E}^{\mydot}) \rightarrow h^{-i}(\omega_{X}^{\mydot}),
\]
had kernel with $P$ as a minimal associated prime, then the new associated map
\[
\left(h^{-i}(\myR \pi_* \omega_E^{\mydot})) \tensor \O_H \cong h^{-i+1}(\myR \pi_* \omega_{E_H}^{\mydot})\right) \rightarrow
\left(h^{-i}(\omega_{X}^{\mydot}) \tensor \O_H \cong h^{-i+1}(\omega_{H}^{\mydot}) \right)
\]
has kernel with the components of $P \cap H$ among the minimal associated primes (since $H$ was chosen to be general).  Note also that $H$ localized at each of the components of $P \cap H$ has Du Bois singularities outside of $P \cap H$ (that is, on the punctured spectrum).  But now our inductive hypothesis guarantees that this is impossible, proving the
proposition.
\end{proof}

\begin{remark}
This also allows us to make the following reduction.  Suppose that $P$ is a minimal prime of the non-Du Bois locus of a variety $X =
\Spec R$.  Then, by inverting a single element $f$ of $R \setminus P$ and replacing $R$ by $R[f^{-1}]$, we may assume that $P$ is the unique
minimal prime of the non-Du Bois locus and, furthermore, that $h^{-i}(\myR \pi_* \omega_{E}^{\mydot}) \rightarrow
h^{-i}(\omega_{X}^{\mydot})$ injects for every $i$.  This is done without completely reducing to the case of a local ring.  Therefore, this can be
thought of as additional progress towards something like a Grauert-Riemenschneider vanishing for Du Bois singularities.  It would be
interesting to know if the above injectivity holds \emph{without} the need to invert any elements; also see
Question \ref{GRVanishingForDuBois}.
\end{remark}

\section{The main theorem and reduction to characteristic $p$}
\label{ReductionToCharacteristicP}

\begin{theorem}
\label{TheoremFInjectiveImpliesDuBois}
Suppose that $X$ is a reduced scheme of finite type over a field of characteristic zero.  If $X$ dense $F$-injective type, then $X$ has
Du Bois singularities.
\end{theorem}

\begin{remark}
The definition of \emph{dense $F$-injective type} is given in Definition \ref{FBlankDefinition}.
\end{remark}

Let us now sketch the proof of Theorem \ref{TheoremFInjectiveImpliesDuBois}.  We will discuss the process of reduction to characteristic $p$
later in this section.  The rest of the proof can then be broken down into three steps.

\begin{itemize}
\item[(i)]  We will show in Theorem \ref{FrobeniusKillsDuBois} that, after generically reducing to characteristic $p$, the
higher cohomology of $\DuBois{X}$ is annihilated by a sufficiently high power of Frobenius.
\item[(ii)]  In the previous section, we proved that if $\Spec R_P \setminus P$ is Du Bois, then the natural map \[H^i_P(R_P) \rightarrow \bH^i_P(X,
\DuBois{X_P})\] is surjective.  In Proposition \ref{PropositionReducingOurInjectivity}, we will show that the same surjectivity can also be
obtained after generic reduction to positive characteristic.
\item[(iii)]  We finally apply a spectral sequence argument which will show that the natural map
\[
h^i(\DuBois{X_P}) \rightarrow H^{i+1}_P(R_P)
\]
is injective for $i \geq 1$.  Frobenius acts injectively on the right, and annihilates the left, proving that $h^i(\DuBois{X_P}) = 0$
for $i > 0$.
\end{itemize}
This then allows us to conclude that our original variety in characteristic zero had Du Bois singularities as well.

Before actually proving this theorem, let us first state a corollary.

\begin{corollary}
\label{CorollaryFPureImpliesDuBois}
Suppose that $X$ is a scheme of finite type over a field of characteristic zero of dense $F$-pure type.  Then $X$ has Du Bois
singularities.
\end{corollary}

\begin{remark}
The key point here is that we do \emph{not} need to assume that $X$ is $\bQ$-Gorenstein or normal.
\end{remark}

Fedder's criterion is a powerful characterization of $F$-purity which is often easily computable (especially if $I$ is a principal ideal).  Suppose that $(S, m)$ be a regular
local ring of characteristic $p$ and let $R = S/I$.  Then $R$ is $F$-pure if and only if $(I^{[p]} : I) \nsubset m^{[p]}$; see
\cite[1.12]{FedderFPureRat}.

\begin{remark}
In his thesis, Davis Doherty used Corollary \ref{CorollaryFPureImpliesDuBois} and Fedder's criterion to prove that certain (non-normal) schemes
had Du Bois singularities; see \cite{DohertySingularitiesOfGenericProjectionHypersurfaces}.
\end{remark}

We now lay out the basics of reduction to characteristic $p$.  An excellent and far more complete reference is
\cite[2.1]{HochsterHunekeTightClosureInEqualCharactersticZero}.  Similar reductions have also been used when establishing links between
the other characteristic zero and $F$-singularities mentioned in the introduction.

Let $R$ be a finitely generated algebra over a field $k$ of characteristic zero.  We can write $R = k[x_1, \ldots, x_n]/I$ for some
ideal $I$ and let $S$ denote $k[x_1, \ldots, x_n]$.  Let $X = \Spec R$, $Y = \Spec S$, and note that we can certainly assume that the
codimension of $X$ in $Y$ is greater than $1$ if desired.  Let $\pi : B_J(Y) = \tld Y \rightarrow Y$ be a strong (projective) log
resolution of $X$ in $Y$ corresponding to the blow-up of an ideal $J$, which has exceptional divisor $E$ (where $\O_{\tld Y}(-E) = (J
\O_{\tld Y})_{\red}$) mapping to $X$.

There exists a finitely generated $\bZ$ algebra $A \subset k$ (including all the coefficients of the generators of $I$ and
$J$), a finitely generated $A$ algebra $R_A \subset R$, an ideal $J_A \subset R_A$, and schemes $\tld Y_A$ and $E_A$ of finite type
over $A$ such that $R_A \tensor_A k = R$, $J_A R = J$, $Y_A \times_{\Spec A} \Spec k = Y$, and $E_A \times_{\Spec A} \Spec k = E$ with
$E_A$ effective and with support equal to the blow-up of $J_A$.  We may localize $A$ at a single element so that $Y_A$ is smooth over
$A$ and $E_A$ is a simple normal crossings divisor over $A$, if desired.  By further localizing $A$ (at a single element), we may assume
any finite set of finitely generated $R_A$ modules is $A$-free (see \cite[3.4]{HunekeTightClosureBook} or
\cite[2.3]{HochsterRobertsFrobeniusLocalCohomology}) and we may assume that $A$ itself is regular.  We may also assume that a fixed
affine cover of $E_A$ and a fixed affine cover of $\tld Y_A$ are also $A$-free.

We will now form a family of positive characteristic models of $X$ by looking at all the rings $R_t = R_A \tensor_A k(t)$ where $k(t)$
is the residue field of a maximal ideal $t \in T = \Spec A$.  Note that $k(t)$ is a finite, and thus perfect, field of characteristic
$p$.  We may also tensor the various schemes $X_A$, $E_A$, etc. with $k(t)$ to produce a characteristic $p$ model of an entire
situation.

By making various cokernels of maps free $A$-modules, we may also assume that maps between modules that are surjective (respectively
injective) over $k$ correspond to surjective (respectively injective) maps over $A$, and thus surjective (respectively injective) in
our characteristic $p$ model as well; see \cite{HochsterHunekeTightClosureInEqualCharactersticZero}.

Various properties of rings that we are interested in descend well from characteristic zero.  For example, smoothness, normality, being
reduced, and being Cohen-Macaulay all descend well \cite[Appendix 1]{HunekeTightClosureBook}.  Specifically, $R_t$ has one of the above
properties for an open set of maximal ideals of $A$ if and only if $R_{(\Frac A)}$ has the same property (in which case so does $R$).
In this spirit, we will need the following lemma.

\begin{lemma}
\label{DenseSeminormalityImpliesSeminormality}
Suppose that $R$ is reduced and of finite type over a field of characteristic zero and that $R_t$ is a family of positive
characteristic models as above (where $A$ is sufficiently large).   If $R_t$ is seminormal for a Zariski dense set of choices of maximal $t \in \Spec A$, then $R$ is
seminormal.
\end{lemma}

This was done for domains in \cite[5.31]{HochsterRobertsFrobeniusLocalCohomology}.
\begin{proof}
Suppose that $R$ is not seminormal, then there are elements $a, b \in R$ such that $a^3 = b^2$ and no $d \in R$ with $d^2 = a$ and $d^3 = b$.  Therefore there exists an elementary subintegral extension $R \subset S = (R[x]/(x^2 - a, x^3 - b))_{\red}$ inside the normalization of $R$; see \cite{SwanSeminormality}.  Let $c$ denote the image of $x$ in $S$ and note that $c \notin R$ by assumption so that $R$ is a proper subset of $S$.  We then reduce this inclusion to an extension of models
\begin{equation}
\label{ElementarySubintegralExtensionModel}
R_A \subset S_A,
\end{equation}
of rings of finite type over $A$, where $A$ is a sufficiently large finitely generated $\bZ$-algebra (as above) and such that $R_A \subset R$ and $S_A \subset S$ in a compatible way (in particular, $S_A$ is reduced since $S$ is).  We localize $A$ so that the cokernel of \ref{ElementarySubintegralExtensionModel} is $A$-free (note the extension must be proper since it is proper in characteristic zero).  We can certainly arrange things so that there still exists $c_A \in S_A \setminus R_A$ which is identified with $c \in S$ such that $c_A^2 = a_A$, $c_A^3 = b_A$.  Note that in particular $c_A$ has non-zero image $\overline c_A$ in $S_A / R_A$.

We now choose a generic maximal ideal $t$ in $A$ satisfying the properties that $R_t = R_A \tensor A/t$ is seminormal and that $S_t = S_A \tensor A/t$ is reduced.  This gives us an inclusion of rings $R_t \subset S_t$ and a non-zero element $c_A \tensor 1 = c_t \in S_t$.  The element $\overline c_t \in S_t/ R_t \cong (S_A / R_A) \tensor A/t$ is non-zero by \cite[2.3(c)]{HochsterRobertsFrobeniusLocalCohomology} and so $c_t$ has no pre-image $R_t$.  However, $c_t^2, c_t^3 \in R_t$, contradicting the seminormality of $R_t$ (or the fact that $S_t$ is reduced).
\end{proof}

The following lemma is very useful for reducing cohomology to prime characteristic.  A sketch of the proof can be found in \cite{HaraRatImpliesFRat}.
\begin{lemma}\cite[4.1]{HaraRatImpliesFRat}
\label{HaraLemma}
Let $X$ be a noetherian separated scheme of finite type over a noetherian ring $A$, and let $\sF$ be a quasi-coherent sheaf on $X$,
flat over $A$.  Suppose that $H^i(X, \sF)$ is a flat $A$-module for each $i > 0$.  Then one has an isomorphism
\[
H^i(X, \sF) \tensor_A k(t) \cong H^i(X_{k(t)}, \sF_{k(t)})
\]
for every point $t \in T = \Spec A$ and $i \geq 0$, where $k(t)$ is the residue field of $t \in T$, $X_{k(t)} = X \times_T
\Spec(k(t))$, and $\sF_{k(t)}$ is the induced sheaf on $X_{k(t)}$.
\end{lemma}

\begin{remark}
In particular, given a map of schemes over a field of characteristic zero, $\pi : E \rightarrow X$, and a corresponding family of maps
of schemes, $\pi_t : E_t \rightarrow X_t$ as above, the previous lemma and (faithfully) flat base change imply that $h^i(\myR \pi_*
\O_E) = 0$ for some $i$ if and only if $h^i(\myR ({\pi_t})_* \O_{E_t}) = 0$ for the same $i$ and all but a finite number of $t$.
\end{remark}

\begin{definition}
\label{FBlankDefinition}
Given a class of singularities, $F$-$\blank$ (such as $F$-injective, $F$-pure, $F$-rational, $F$-regular) in characteristic $p$, we say
that a ring $R$ of finite type over a field of characteristic zero has \emph{open $F$-$\blank$ type} if for a single (equivalently, every sufficiently large) choice of $A$ as
above, $R_t$ is $F$-$\blank$ for an open set of maximal ideals $t \in \Spec A$.  We say that such a ring $R$ has \emph{dense
$F$-$\blank$ type} (or simply \emph{$F$-$\blank$ type}) if $R_t$ is $F$-$\blank$ for a Zariski-dense set of maximal ideals $t \in \Spec
A$.
\end{definition}

Finally we should note that extending the ring $A$ by a finite number of additional scalars from $k$ (and thus also adding those
scalars to $R_A$) does not change whether a dense set of models is $F$-injective.  Extending $A$ by additional scalars and modding out
by a new maximal ideal $t \in A$ simply yields a model which extends the finite field we are working with (from some original model).
Flat base change for local cohomology makes it easy to see that this does not affect the property of $F$-injectivity since all these
field extensions are of finite fields.  This technique has been used for other classes of singularities as well, such as $F$-rational,
$F$-regular and $F$-pure, and this fact has been explicitly stated for $F$-injective singularities in
\cite{MehtaSrinivasRatImpliesFRat}.

\section{Proof of the main theorem}
\label{SectionProofOfMain}

The first goal of this section is to show that the action of Frobenius annihilates the higher cohomology of $\DuBois{X}$ after an
appropriate reduction to characteristic $p$.  This is perhaps not unexpected since rational singularities exhibit a similar behavior.

Suppose that $(R, m)$ is a normal local ring reduced generically from characteristic zero to $p$ with a resolution $\tld X \rightarrow
\Spec R$ satisfying the condition that $\Spec R \setminus m$ has rational singularities.  Then $H^{d-1}(\tld X, \O_{\tld X})$ naturally injects
into $H^d_m(R)$ by the dual Grauert-Riemenschneider vanishing.  Furthermore, that injection naturally identifies $H^{d-1}(\tld X,
\O_{\tld X})$ with $0_{H^d_m(R)}^*$, the tight closure of zero in $H^d_m(R)$; see \cite[4.7]{HaraRatImpliesFRat},
\cite[7.3]{MehtaSrinivasRatImpliesFRat}, and \cite[5.4]{HaraInterpretation}.  In the Du Bois case, we have perhaps an even more
interesting behavior.  Without reference to any local cohomology, Frobenius naturally acts on the higher cohomology of the Du Bois
complex, and we will show that Frobenius annihilates it.

\begin{theorem}
\label{FrobeniusKillsDuBois}
Suppose that $R$ is a ring of finite type over a perfect field $k$ of characteristic $p$.  Further suppose that $\xymatrix{k[x_1,
\ldots, x_n] \ar@{->>}[r] & R,}$ corresponds to a (codimension 2 or greater) embedding of $X = \Spec R$, which was reduced generically
from characteristic zero with a strong (projective) log resolution of $X$, $\pi : \tld Y \rightarrow Y = \Spec k[x_1, \ldots, x_n]$.
Let $E$ be the reduced pre-image of $X$ in $\tld Y$.  Then
\[
F^e(h^i(\myR \pi_* \O_E)) = F^e(h^i(\DuBois{X})) = 0
\]
for all $i > 0$ and all sufficiently large $e$.
\end{theorem}
\begin{proof}
Using the alternate Du Bois construction of Theorem \ref{AlternateDuBoisCharacterization}, this proof is actually remarkably easy!  We have a
natural isomorphism $h^i(\myR \pi_* \O_E) \cong h^{i+1}(\myR \pi_* \O_{\tld Y}(-E))$ for $i > 0$ since the resolution $\pi : \tld Y \rightarrow Y$ was reduced generically from
characteristic zero where it had rational singularities.  There is an anti-effective $\pi$-ample divisor $D$ with support equal to $E$
corresponding to the ideal that was blown up.  Now, the Frobenius action on $E$ and $\tld Y$ is somewhat subtle.  We have the
following factorization for $q = p^e$ and $i > 0$
\[
\xymatrix{
h^i(\myR \pi_* \O_E) \ar@{=}[r] \ar[d] & h^{i+1}(\myR \pi_* \O_{\tld Y}(-E) \ar[d]^{F^e})\\
h^i (\myR \pi_* \O_{qE}) \ar@{=}[r] \ar[d] & h^{i+1}(\myR \pi_* \O_{\tld Y}(-qE) \ar[d]^{\rho}) \\
h^i (\myR \pi_* \O_{E}) \ar@{=}[r] & h^{i+1}(\myR \pi_* \O_{\tld Y}(-E)) \\
}
\]
where the composition of the left column is the Frobenius action on $E$, and the map $\rho$ is just the one associated to the inclusion
$\O_{\tld Y}(-qE) \rightarrow \O_{\tld Y}(-E)$.  Note that all these maps are compatible.  We can now choose $n$ large
enough so that $h^{i+1}(\myR \pi_* \O_{\tld Y}(-nD)) = 0$ by Serre Vanishing, \cite[2.2.1]{EGAIII1}.  However, we can pick $q = p^e$
much larger so that we obtain the following factorization:
\[
\O_{\tld Y}(-qE) \subset \O_{\tld Y}(-nD) \subset \O_{\tld Y}(-E)
\]
This means that the composition of the right column in the diagram is zero for $q$ sufficiently large, and thus $F^e(h^i(\myR \pi_*
\O_E)) = 0$ as desired.
\end{proof}

As stated before, our next goal is to force the higher cohomology of the Du Bois complex to embed into local cohomology.
However, to do that we first need a lemma which will be used to reduce properties of a dualizing complex from a variety over $\bC$ to a
model over $A$ (where $A$ is a finitely generated $\bZ$ algebra); also see \cite{SmithFRatImpliesRat}.

\begin{lemma}
\label{LemmaExtBaseChange}
Suppose that $A$ is a finitely generated $\bZ$-algebra that is a regular ring.  Let $S_A = A[x_1, \ldots, x_n]$ and note that $S_A$ is
Gorenstein and we may take $\omega_{S_A}^{\mydot}$ to be $S_A$ (or with a shift if desired).  If $A \subset B$ with $B$ a field (so in
particular it is a flat $A$-algebra) then for any bounded complex of $S_A$-modules $N^{\mydot}$ with coherent cohomology, we have
\[
(\myR \Hom_{S_A}^{\mydot}(N^{\mydot}, S_A)) \tensor_A B \qis (\myR \Hom_{S_A \tensor_A B}^{\mydot}(N^{\mydot} \tensor_A B, S_A
\tensor_A B)).
\]
In particular, if we replace $N^{\mydot}$ by $R_A$, a quotient of $S_A$, then using the characterization of $f^!$ for finite maps, we
see that \[(\omega_{R_A}^{\mydot}) \tensor_A B \qis \omega_{(R_A \tensor_A B)}^{\mydot}.\]  Note that this quasi-isomorphism is not
typically going to be of normalized dualizing complexes; however, up to a shift (by $[\dim A]$), it will be.
\end{lemma}
\begin{proof}
This follows from standard techniques involving $\myR \Hom^{\mydot}$ and from the fact that $\Hom_{S_A}(N, S_A) \tensor_A B
\cong \Hom_{S_A \tensor_A B}(N \tensor_A B, S_A \tensor_A B)$ for finite $S_A$-modules $N$.
\end{proof}

\begin{remark}
Obviously the previous statement holds under substantially more general hypotheses.
\end{remark}

The following proposition is the last major step in the proof of Theorem \ref{TheoremFInjectiveImpliesDuBois}.

\begin{proposition}
\label{PropositionReducingOurInjectivity}
Suppose that $X_\bC = \Spec R$ is of finite type over $\bC$ and that $X_\bC \subset Y_{\bC} = \Spec \bC[x_1, \ldots, x_n]$.  Suppose
that $\pi_{\bC} : \tld Y_{\bC} \rightarrow Y_{\bC}$ is a strong log resolution of $X_{\bC}$ in $Y_{\bC}$ and that $E_{\bC} =
(\pi_{\bC}^{-1}(X_{\bC}))_{\red}$.  If the natural morphism
\begin{equation}
\label{EquationToReduceFromKtot}
h^i(\myR (\pi_{\bC})_* \omega_{E_{\bC}}^\mydot) \rightarrow h^i( \omega_{X_{\bC}}^\mydot)
\end{equation}
is injective for all $i$, then for a generic choice of positive characteristic model, $\pi_t : E_t \rightarrow X_t$, the following
natural morphisms are injections for all $i$,
\[
\xymatrix{
h^i(\myR (\pi_t)_* \omega_{E_{t}}^\mydot) \ar@{^{(}->}[r] & h^i(\omega_{X_{t}}^\mydot).
}
\]
\end{proposition}
\begin{proof}
We now choose an appropriate finitely generated $\bZ$-algebra $A \subset \bC$ and reduce our data to this situation.  We let $S_A$ be
the ring of polynomials in $m$ variables over $A$, of which $R_A$ is a quotient.  We keep track of the corresponding map $\pi_A
: \tld Y_A \rightarrow Y_A$ and the reduced pre-image of $X_A$, which we call $E_A$.  Note that $E_A \times_A \bC \cong E_{\bC}$.  It
is easy to see that we may choose $\pi_A$ so that it is still obtained by blowing up an ideal whose total transform has support equal
to $E_A$.  Without loss of generality we may assume that $A$ is regular.

We first wish to use the injectivity \ref{EquationToReduceFromKtot}, to obtain the injectivity
\[
\xymatrix{
h^i(\myR (\pi_{A})_* \omega_{E_{A}}^{\mydot}) \ar@{^{(}->}[r] & h^i(\omega_{X_{A}}^{\mydot})
}
\]
(at least after replacing $A$ by a suitable localization).  Consider the map
\begin{equation}
\label{DefiningMap}
R_A \rightarrow \Gamma(X_A, \myR (\pi_{A})_* \O_{E_A}).
\end{equation}
Since $X_A$ is affine, we will abuse notation and leave out the functor $\Gamma$ in the future.  We may represent $\myR (\pi_{A})_* \O_{E_A}$
by \Cech cohomology which behaves well with respect to flat base change (for example, see \cite[III.9.3]{Hartshorne}).
We now apply the functor $\myR \Hom_{S_A}^{\mydot}(\blank, S_A[m])$ to \ref{DefiningMap}.  Grothendieck duality and the description of
$f^!$ for finite maps lets us identify the map
\begin{equation}
\label{EquationToReduceFromKtoA}
\myR \Hom_{S_A}^{\mydot}(\myR (\pi_{A})_* \O_{E_A}, S_A[m]) \rightarrow \myR \Hom^{\mydot}_{S_A}(R_A, S_A[m])
\end{equation}
with the map
\[
\myR (\pi_{A})_* \omega_{E_A}^{\mydot} \rightarrow \omega_{X_A}^{\mydot}
\]
after shifting by the dimension of $A$.  It is the characterization of the map in
\ref{EquationToReduceFromKtoA} that we will reduce from over $\bC$ to $A$.  We now localize $A$ to make the various modules mentioned
above (in particular, the cohomology of the above complexes) the modules corresponding to a set of affine charts on $E_A$, and the various kernels and cokernels of maps between the various
modules all $A$-free.  Finally, we apply Lemma \ref{LemmaExtBaseChange} and we see that
\[
\xymatrix{
h^i(\myR (\pi_{A})_* \omega_{E_{A}}^{\mydot}) \ar@{^{(}->}[r] & h^i(\omega_{X_{A}}^{\mydot})
}
\]
injects for each $i$ as desired.

We now will show that the injectivity
\[
\xymatrix{
h^i(\myR (\pi_{t})_* \omega_{E_{t}}^{\mydot}) \ar@{^{(}->}[r] & h^i(\omega_{X_{t}}^{\mydot})
}
\]
is preserved for each $i$ and for a general choice (in fact at this point, any choice) of maximal $t \in \Spec A$.  Therefore choose a maximal $t \in \Spec A$ arbitrarily.

Now note that $R_t$ (respectively $E_t$) is a closed subscheme of $R_A$ (respectively $E_A$).
Apply $f^!(\blank) = \myR \Hom^{\mydot}_{R_A}((R_A)/(t
R_A), \blank)$ to the map $\myR (\pi_{A})_* \omega_{E_{A}}^{\mydot} \rightarrow \omega_{X_{A}}^{\mydot}$.  By Grothendieck duality, this gives us the map in whose cohomology we are ultimately interested, $\myR (\pi_{t})_* \omega_{E_{t}}^\mydot \rightarrow \omega_{X_{t}}^\mydot$.  We wish to analyze the induced maps on cohomology by analyzing the spectral sequences (and the induced map between these spectral sequences)
\begin{align*}
h^p \myR \Hom_{R_A}^{\mydot}(R_A/(t R_A),  h^q \myR (\pi_{A})_* \omega_{E_{A}}^{\mydot}) & \Rightarrow  h^{p+q}\myR \Hom^{\mydot}_{R_A}(R_A/(t
R_A),  \myR (\pi_{A})_* \omega_{E_{A}}^{\mydot}) \\
h^p \myR \Hom_{R_A}^{\mydot}(R_A/(t R_A),  h^q \omega_{X_{A}}^{\mydot}) & \Rightarrow h^{p+q} \myR \Hom^{\mydot}_{R_A}(R_A/(t
R_A),  \omega_{X_{A}}^{\mydot})
\end{align*}

Suppose that $M$ is an arbitrary $R_A$-module that is free as an $A$-module.  I claim that $N = \Ext^i_{R_A}(R_A/(t R_A), M) = 0$ for $i \neq c = \height t$.  To see this, notice that the $R_A$-module $N$ is zero if and only if it is zero when viewed as an $A$-module.  On the other hand, the module $N$ is annihilated by $t$.  Let us use $U$ to denote the multiplicative system $A \setminus t$ and let $B$ be the regular local ring $U^{-1} A$.  Finally set $R_B = U^{-1} R_A \cong R_A \tensor_A B$.  Then notice that
\[
N \cong N \tensor_A B \cong (N \tensor_{R_A} R_A) \tensor_A B \cong N \tensor_{R_A} R_B \cong \Ext^i_{R_B}(R_B/(t R_B), M \tensor_A B).
\]
Now recall that $B \subset R_B$ is a regular local ring with maximal ideal $t B$ and that the modules $M \tensor_A B$ and $R_B$ are free $B$-modules (also observe that $R_B/(t R_B) \cong R_A /(t R_A)$).  But then note $t R_B$ is generated by a regular sequence for the module $M \tensor_A B$ (it also forms a regular sequence for the module $R_B$).  This implies that $\Ext^i_{R_B}(R_B/(t R_B), M \tensor_A B) = 0$ for $i \neq c = \height t = \height (t B)$, which proves the claim.

Therefore the spectral sequences mentioned above are collapsed (single column) spectral sequences.  Thus we see that the map
\[
\xymatrix{
h^i(\myR (\pi_t)_* \omega_{E_{t}}^\mydot) \ar[r] & h^i(\omega_{X_{t}}^\mydot)
}
\]
is identified with the map
\[
\xymatrix{
\Ext_{R_A}^{c}(R_A/(t R_A), h^{i-c} \myR (\pi_{A})_* \omega_{E_{A}}^{\mydot}  ) \ar[r] & \Ext_{R_A}^{c} (R_A/(t R_A), h^{i-c} \omega_{X_{A}}^{\mydot}).
}
\]
Using the fact that we made the
cokernels of the maps
\[
\xymatrix{
h^{i-c} \myR (\pi_{A})_* \omega_{E_{A}}^{\mydot} \ar@{^{(}->}[r] & h^{i-c} \omega_{X_{A}}^{\mydot}
}
\]
$A$-free, we see that we achieve our desired injectivity.
\end{proof}

We now complete the proof of Theorem \ref{TheoremFInjectiveImpliesDuBois}

\begin{proof}
Furthermore, we may assume that $X_{k} =
\Spec R_{k}$ is affine of dimension $d$ and embeds in a smooth $Y_{k} = \bA^m_{k} = \Spec S_{k}$ with codimension at least two.  Let $\pi_{k} : \tld Y_{k} \rightarrow Y_{k}$ be a strong (projective) log resolution of $X_{k} \subset Y_{k}$ and let $E_{k}$ be the
reduced pre-image of $X_{k}$.  Let $\sI_{k}$ be the ideal blown up to obtain this resolution.

Now I claim that we may assume that $k = \bC$.  To check this it is enough to show that the properties of having dense $F$-injective type and having Du Bois singularities are invariant under the requisite base change operations.  Suppose then that $k$ is a field of characteristic zero and $X = X_k$ is of finite type over $k$.  Suppose further that $K \subseteq k$ is isomorphic to a subfield of $\bC$ and that $X_k$, $Y_k$, $E_k$ etc. are defined over $K$.  Abusing notation, we identify $K$ with that subfield of $\bC$.  Then $X_K$ is Du Bois if and only if $\O_{X_K} \rightarrow \DuBois{X_K}$ is a quasi-isomorphism.  But that map is a quasi-isomorphism if and only if it is a quasi-isomorphism after tensoring with $\bC$ (or with $k$).  On the other hand, $X_k$ has dense $F$-injective type if and only if $X_K$ has dense $F$-injective type if and only if $X_{\bC}$ has dense $F$-injective type (one may even use the same set of characteristic $p$ models).  Thus it is harmless to assume that $k = \bC$.

Note that being of dense $F$-injective
type implies seminormality by Theorem \ref{FInjectiveImpliesWeaklyNormal} and Lemma \ref{DenseSeminormalityImpliesSeminormality}.  Therefore we have
$\pi_* \O_{E_k} \cong \O_{X_{k}}$ and so we only need to show that
\[
h^i(\myR (\pi_{k})_* (\O_{E_{k}}))_{P_{k}} = 0
\]
for $i > 0$.
Suppose that $X$ is not Du Bois, so let $P_{k}$ be a prime of $R_{k}$, minimal with respect to the condition that $(R_{k})_{P_{k}}$ is
\emph{not} Du Bois.  By localizing appropriately we can, without loss of generality, assume that there is an injection
\begin{equation}
\label{CharacteristicZeroInjectivity}
\xymatrix{
h^i(\myR (\pi_{k})_* \omega_{E_{k}}^{\mydot}) \ar@{^{(}->}[r] & h^i(\omega_{X_{k}}^{\mydot})
}
\end{equation}
since we can apply Proposition \ref{GeneralizedKovacsSurjectivity} and since inverting a single element of $R_{k}$ will preserve dense
$F$-injective type.  Note in particular that we still have $h^i(\myR (\pi_k)_* \O_{E_k}) \neq 0$ for some $i > 0$.

As before, we choose an appropriate finitely generated $\bZ$-algebra $A \subset \bC$ and reduce our entire situation to finite type
over $A$.  We also localize $A$ so that all $R_A$-modules we encounter are $A$-free (in particular, so that the conclusions of
Theorem \ref{FrobeniusKillsDuBois} and Proposition \ref{PropositionReducingOurInjectivity} will hold).  Note we may assume that $(\pi_A)_* \O_{E_A}
\cong \O_{X_A}$.

Choose a maximal $t \in \Spec A$ such that $R_t = R_A \tensor_A A/t$ is $F$-injective and tensor $R_A$ and $E_A$ with $A/t$.  In
particular, we preserve the isomorphism $R_t \cong (\pi_t)_* \O_{E_t}$.   Choose $Q_t$ a prime minimal with respect to the condition
that $h^i(\myR(\pi_t)_* \O_{E_t})_{Q_t} \neq 0$ for some $i > 0$ and localize $R_t$ to form $(R_t)_{Q_{t}}$.  Such a $Q_t$ exists since
$h^i(\myR(\pi_t)_* \O_{E_t})$ cannot be equal to zero, for all $i > 0$ (this is because $h^i(\myR (\pi_{k})_* (\O_{E_{k}}))$ is not
equal to zero for all $i$; see Lemma \ref{HaraLemma}).  Note, $(R_t)_{Q_t}$ is $F$-injective by assumption.

An application of local duality applied to the conclusion of Proposition \ref{PropositionReducingOurInjectivity} obtains the following surjectivity
for each $i$ (note that there is a sign flip and possible shift on $i$ that we are suppressing):
\begin{equation}
\label{ReducedSurjectivity}
\xymatrix{
H^i_{Q_t}((R_t)_{Q_t}) \ar@{->>}[r] & \bH^i_{Q_t}( (\myR (\pi_t)_* \O_{E_t})_{Q_t} ).
}
\end{equation}
However, $h^i((\myR (\pi_t)_* \O_{E_t})_{Q_t})$  is supported at $Q_t$ for $i > 0$, and so we consider the spectral sequence computing
$\myR \Gamma_{Q_t}((\myR (\pi_t)_* \O_{E_t})_{Q_t})$.  In particular, we have the spectral sequence
\[
h^p(\myR \Gamma_{Q_t}(h^q(\myR (\pi_t)_* \O_{E_t})_{Q_t})) \Rightarrow h^{p + q}(\myR \Gamma_{Q_t}((\myR (\pi_t)_* \O_{E_t})_{Q_t})) = \bH^{p+q}_{Q_t}((\myR (\pi_t)_* \O_{E_t})_{Q_t}).
\]

Since the modules $h^i((\myR (\pi_t)_* \O_{E_t})_{Q_t})$ have support contained
in the maximal ideal of $(R_t)_{Q_t}$ for $i > 0$, the terms $E^{pq}$ of this spectral sequence are zero unless either $p$ or $q$ are zero.  This
implies that this spectral sequence caries the data of a long exact sequence:
\[
\xymatrix@R=8pt{
0 \rightarrow H^1_{Q_t}(((\pi_t)_* \O_{E_t})_{Q_t}) \ar[r] & \bH^1_{Q_t}( (\myR (\pi_t)_* \O_{E_t})_{Q_t} ) \ar[r] & h^1(\myR(\pi_{t})_*
\O_{E_{t}})_{Q_t} \ar[r] & \dots \\
\dots \rightarrow H^i_{Q_t}(((\pi_t)_* \O_{E_t})_{Q_t}) \ar[r] & \bH^i_{Q_t}( (\myR (\pi_t)_* \O_{E_t})_{Q_t} ) \ar[r] &
h^i(\myR(\pi_{t})_* \O_{E_{t}})_{Q_t} \ar[r] & \dots .
}
\]
Therefore, by noting that $(R_t)_{Q_t} \cong ((\pi_t)_* \O_{E_t})_{Q_t}$, and applying the surjectivity \ref{ReducedSurjectivity}, we
obtain the following injectivity for each $i>0$,
\[
\xymatrix{
h^i(\myR(\pi_{t})_* \O_{E_{t}})_{Q_t} \ar@{^{(}->}[r] & H^{i+1}_{Q_t}((R_t)_{Q_t}).
}
\]
The action of Frobenius annihilates the left-hand-side and acts injectively on the right-hand-side, which implies that $h^i
(\myR(\pi_{t})_* \O_{E_{t}})_{Q_t} = 0$ for $i > 0$.  This contradicts our choice of $Q_t$, and thus also contradicts the existence of $P_k$, which completes the proof.
\end{proof}

\section{Further questions}

It is still conjectural that log canonical singularities have dense $F$-pure type, and likewise the corresponding statement between Du
Bois and $F$-injective singularities is also open.

\begin{question}
Suppose that $X$ has Du Bois singularities, does $X$ have dense $F$-injective type?
\end{question}

There is a certain amount of evidence towards this question besides the results in this paper.  In particular, there are several other
properties that Du Bois singularities and $F$-injective singularities share; see \cite{KarlThesis}.  We also might hope for a result
similar to what occurs for rational singularities.

\begin{question}
Suppose that $X_t$ is a family of models reduced from a characteristic zero scheme $X$ which had an embedding into a smooth variety $X
\subset Y$ and a strong log resolution $\pi$ of $X$ with $E$ the reduced pre-image of $X$ (so that $\myR \pi_* \O_E \qis \DuBois{X}$
with the usual notations).  Suppose $X$ had an isolated point $x$ such that $X \setminus x$ is Du Bois.  Is it true that, for a dense set of
models, $h^i(\myR(\pi_t)_* \O_{E_t})_{x_t}$ is naturally identified with the Frobenius closure of $0$ in $H^{i+1}_{x_t}(X_t,
\O_{X_t})$?
\end{question}

The following injectivity, which, if it were known to be true, could be thought of as analogous to the Grauert-Riemenschneider
vanishing theorem, and would be useful in certain problems related to deformations of Du Bois singularities.

\begin{question}
\label{GRVanishingForDuBois}
Suppose that $X$ embeds into a
smooth scheme $Y$ of finite type over a field of characteristic zero, and that $\pi : \tld Y \rightarrow Y$ is a strong log resolution
of $X$ with reduced pre-image $E = (\pi^{-1} X)_{\red}$.
Then is it true that for each $i$, the natural map
\[
h^i(\myR \pi_* \omega_E^{\mydot}) \rightarrow h^i(\omega_X^{\mydot})
\]
is an injection?
\end{question}

If $X$ is Cohen-Macaulay, this would imply that $\myR \pi_* \omega_E^{\mydot}$ has trivial cohomology except in one degree, which more
closely resembles the classical Grauert-Riemenschneider vanishing theorem.


\providecommand{\bysame}{\leavevmode\hbox to3em{\hrulefill}\thinspace}
\providecommand{\MR}{\relax\ifhmode\unskip\space\fi MR}
\providecommand{\MRhref}[2]{%
  \href{http://www.ams.org/mathscinet-getitem?mr=#1}{#2}
}
\providecommand{\href}[2]{#2}

\end{document}